\documentclass[a4paper,11pt]{amsproc}
\usepackage{amsmath,amsfonts,amssymb,amsthm,anysize}
\usepackage{enumerate}
\usepackage{epsfig}
\usepackage{supertabular}
\usepackage{graphicx}

\newcommand\N{{\mathbb N}}
\newcommand\Z{{\mathbb Z}}
\newcommand\Q{{\mathbb Q}}
\newcommand\R{{\mathbb R}}

\newcommand\C{{\mathbb C}}
\newcommand\F{{\mathbb F}}

\newcommand\Gal{{\mathrm{Gal}}}

\newcommand\Tr{{\mathrm{Tr}}}

\renewcommand\mod{{\mathrm{mod\, \, }}}

\theoremstyle{plain}
\newtheorem{theorem}{Theorem}[section]
\newtheorem{definition}{Definition}[section]

\newtheorem{lemma}[theorem]{Lemma}
\newtheorem{corollary}[theorem]{Corollary}

\newtheorem{proposition}[theorem]{Proposition}

\numberwithin{equation}{section}

\theoremstyle{remark}
\newtheorem{remark}[theorem]{Remark}

\def\tra{{\mathrm{T}}}

\usepackage{url, hyperref}
\hypersetup{citecolor=blue, linkcolor=blue, colorlinks=true}

\renewcommand\le{\leqslant}
\renewcommand\ge{\geqslant}

\begin{document}

\title[Skew Hadamard difference families and skew Hadamard 
matrices]{Skew Hadamard difference families and \\skew Hadamard 
matrices}


\author{Koji Momihara}
\address{ %
Department of Mathematics\\
Faculty of Education\\
Kumamoto University\\
2-40-1 Kurokami, Kumamoto 860-8555, Japan}
\email{momihara@educ.kumamoto-u.ac.jp}
\thanks{Koji Momihara was supported by 
JSPS under Grant-in-Aid for Young Scientists (B) 17K14236 and Scientific Research (B) 15H03636.}

\author{Qing Xiang}
\address{ %
Department of Mathematical Sciences\\
University of Delaware\\
Newark DE 19716, USA
}
\email{qxiang@udel.edu}
\thanks{Qing Xiang was supported by an NSF grant DMS-1600850, and a JSPS invitational fellowship for research in Japan S17114.}

\subjclass[2010]{
}

\begin{abstract}
In this paper, we generalize classical  constructions of skew Hadamard difference families with two or four blocks in the additive groups of finite fields given by Szekeres (1969, 1971), Whiteman (1971) and Wallis-Whiteman (1972). In particular, we show that there exists a skew Hadamard difference family with $2^{u-1}$ blocks in the additive group of the finite field of order $q^e$ for any prime power $q\equiv 2^u+1\,(\mod{2^{u+1}})$ with $u\ge 2$ and any positive  integer $e$. In the aforementioned work of Szekeres, Whiteman, and Wallis-Whiteman, the constructions of skew Hadamard difference families with $2^{u-1}$ ($u=2$ or $3$) blocks in $(\F_{q^e},+)$ depend on the exponent $e$,  with $e\equiv 1,2,$ or $3\,(\mod{4})$ when $u=2$, and $e\equiv 1\,(\mod{2})$ when $u=3$, respectively. Our more general construction, in particular, removes the dependence on $e$. As a consequence, we obtain new infinite families of skew Hadamard matrices. 
\end{abstract}


\maketitle

\section{Introduction}\label{sec:notat}

A Hadamard matrix of order $n$ is an $n\times n$ matrix $H$ with entries $\pm 1$ such that $HH^{\top}=nI_n$, where $I_n$ is the identity matrix of order $n$. 
It is well known that if $H$ is a Hadamard matrix of order $n$ then $n=1,2$ or $n\equiv 0\pmod 4$. One of the most famous conjectures in combinatorics states that a Hadamard matrix of order $n$ exists for every positive integer 
$n$ divisible by $4$. This conjecture is far from being resolved despite extensive research on the problem. The smallest $n$ for which the existence of a Hadamard matrix of order $n$ is unknown is currently $668$ (see \cite{kt}). 
In this paper, we are interested in Hadamard matrices which are ``skew''.  A Hadamard matrix is called {\it skew} if $H=A+I_n$ and $A^{\top}=-A$. 
See \cite{KS} for a short survey of known constructions of skew Hadamard matrices. One of the most effective methods for constructing 
(skew) Hadamard matrices is by using difference families. Let $(G, +)$ be an additively written abelian group of order $v$. A {\it difference family} with parameters $(v,k,\lambda)$ in $G$ is a family ${\mathcal B}=\{B_i\,|\,i=1,2,\ldots,\ell\}$ of $k$-subsets 
of $G$ such that 
the list of differences ``$x-y, x,y\in B_i,x\not=y,i=1,2,\ldots,\ell$" represents every 
nonzero element of $G$ exactly $\lambda$ times. Each subset $B_i$ is called 
a {\it block} of the difference family.  A block $B_i$ is called {\it skew} if it has the property that $B_i\cap -B_i=\emptyset$ and 
$B_i\cup -B_i=G\setminus \{0_G\}$. If all blocks of a difference family are skew, then the difference family is called {\it skew Hadamard}.

We review two known constructions of skew Hadamard matrices based on 
difference families. Let $X$ be a subset of a finite abelian group $(G, +)$. 
Fixing an ordering for the elements of $G$, we define matrices $M=(m_{i,j})$ and $N=(n_{i,j})$ by   
\[
m_{i,j}=\begin{cases}
1,& \text{ if }  j-i\in X, \\
-1, & \text{ if } j-i\not\in X, 
\end{cases}
\mbox{\, \,  and \, \, }n_{i,j}=\begin{cases}
1,& \text{ if }  j+i\in X, \\
-1, & \text{ if } j+i\not\in X.
\end{cases}
\]
The matrices $M$ and $N$ are called {\it type-1} and  {\it type-2} matrices of $X$, respectively. 
\begin{proposition}\label{prop:hada2ee1}{\em (\cite[Theorem~4.4]{WSW})}
Let ${\mathcal B}=\{B_i\,|\,i=1,2\}$ be a difference family with parameters 
$(v,k,\lambda)=(2m+1,m,m-1)$ such that $B_1$ is skew. Furthermore, let  
$M_1$ be the type-1 matrix of $B_1$ and 
$M_2$ be  the type-2 matrix of $B_2$. Then, 
\begin{equation}
H=\begin{pmatrix} 
1  &1 &  {\bf 1}_v^\tra&  {\bf 1}_v^\tra  \\
-1  & 1 & {\bf 1}_v^\tra&-{\bf 1}_v^\tra \\
-{\bf 1}_v  &  -{\bf 1}_v &-M_1 & -M_2 \\
-{\bf 1}_v  & {\bf 1}_v&M_2 & -M_1  
 \end{pmatrix}
\end{equation}
is a skew Hadamard matrix of order $4(m+1)$. 
\end{proposition}
Szekeres~\cite{Sz0,Sz} and Whiteman~\cite{Wh}  found two series of skew Hadamard difference families with two blocks in $(\F_q,+)$,  the additive group of the finite field $\F_{q}$ of order $q$. 
\begin{proposition}\label{prop:skewfour}
There exists a skew Hadamard difference family with two blocks in $(\F_q,+)$ if \begin{itemize}
\item[(i)] {\em (\cite{Sz0})} $q\equiv 5\,(\mod{8})$; or 
\item[(ii)] {\em (\cite{Sz,Wh})} $q=p^e$ with $p\equiv 5\,(\mod{8})$ a prime and $e\equiv 2\,(\mod{4})$. 
\end{itemize}
\end{proposition}
The proofs of the results above are based on 
cyclotomic numbers of 
order four and eight, respectively. 
Szekeres~\cite{Sz} claimed that his proof for Part (ii) of Proposition~\ref{prop:skewfour} works well also for the case where 
$e\equiv 0\,(\mod{4})$. However, in the case where $e\equiv 0\pmod 4$, the two subsets demonstrated in Theorem~1 of \cite{Sz} are not skew. This inconsistency was pointed out in \cite[p.~324]{WSW}, and also in the MathSciNet  mathematical review of \cite{Sz} written by B. M. Stewart. 

\begin{proposition}\label{prop:hada2ee2}{\em (\cite{WW})}
Let ${\mathcal B}=\{B_i\,|\,i=1,2,3,4\}$ be a difference family with parameters 
$(v,k,\lambda)=(2m+1,m,2(m-1))$ such that $B_1$ is skew. Furthermore, let  
$M_1,M_2,M_4$ be the type-1 matrices of $B_1,B_2,B_4$, respectively, and 
$M_3$ be  the type-2 matrix  of $B_3$. Then, 
\begin{equation}
H=\begin{pmatrix} 
1  &-1 & -1  &-1 & -{\bf 1}_v^\tra& - {\bf 1}_v^\tra & -{\bf 1}_v^\tra& - {\bf 1}_v^\tra  \\
1  &1 & 1  &-1 & {\bf 1}_v^\tra& - {\bf 1}_v^\tra & {\bf 1}_v^\tra& - {\bf 1}_v^\tra  \\
1  &-1 & 1  &1 & {\bf 1}_v^\tra& - {\bf 1}_v^\tra & -{\bf 1}_v^\tra&  {\bf 1}_v^\tra  \\
1  &1 & -1  &1 & {\bf 1}_v^\tra&  {\bf 1}_v^\tra & -{\bf 1}_v^\tra&  -{\bf 1}_v^\tra  \\
{\bf 1}_v  &  -{\bf 1}_v &-{\bf 1}_v  & - {\bf 1}_v &-M_1 & -M_2&-M_3&-M_4 \\
{\bf 1}_v  &  {\bf 1}_v &{\bf 1}_v  & - {\bf 1}_v &M_2^T & -M_1^\tra&M_4&-M_3 \\
{\bf 1}_v  & - {\bf 1}_v &{\bf 1}_v  & {\bf 1}_v &M_3 &- M_4^\tra&-M_1&M_2^\tra \\
{\bf 1}_v  &  {\bf 1}_v &-{\bf 1}_v  &  {\bf 1}_v &M_4^\tra & M_3&-M_2&-M_1^\tra \\
 \end{pmatrix}
\end{equation}
is a skew Hadamard matrix of order $8(m+1)$. 
\end{proposition}
Wallis and Whiteman~\cite{WW} found one series of skew Hadamard difference families   with four blocks in $(\F_q,+)$ based on cyclotomic numbers of 
order eight. 
\begin{proposition}\label{prop:skewfour2}
There exists a skew Hadamard difference family with four blocks in $(\F_q,+)$ if  $q\equiv 9\,(\mod{16})$. 
\end{proposition}
In this paper, we generalize the  results in Propositions~\ref{prop:skewfour} and \ref{prop:skewfour2} using cyclotomic classes of order a power of $2$. 
In general, it is quite difficult to find explicit formulas for 
cyclotomic numbers of high order. In this paper, we overcome this difficulty by evaluating Gauss sums with respect to a multiplicative character of order a power of $2$ by a recursive technique.
 
\begin{theorem}\label{thm:main2}
Let $u\ge 2$ be an integer and $q$ be a prime power such that $q\equiv 2^u+1\,(\mod{2^{u+1}})$. 
Then, there exists a skew Hadamard difference family with $2^{u-1}$ blocks in 
$(\F_{q^e},+)$ for any positive integer $e$. 
\end{theorem}

We emphasize that the exponent $e\geq 1$ can be taken arbitrarily in Theorem~\ref{thm:main2}. In contrast, Propositions~\ref{prop:skewfour} and \ref{prop:skewfour2} depend on $e$.  In the case of Proposition~\ref{prop:skewfour}, the exponent $e$ is limited to $e\equiv 1,2,$ or $3\,(\mod{4})$, and in the case of Proposition~\ref{prop:skewfour2},  the exponent $e$ is limited to $e\equiv 1\,(\mod{2})$.

By applying Propositions~\ref{prop:hada2ee1} and \ref{prop:hada2ee2} to the skew Hadamard difference families arising from Theorem~\ref{thm:main2} with $u=2$ and $u=3$, respectively, we have the following corollaries.  
\begin{corollary}
Let $q$ be a prime power such that $q\equiv 5\,(\mod{8})$ and $e$ be an arbitrary positive integer. 
Then, there exists a skew Hadamard matrix of order $2(q^e+1)$. 
\end{corollary}
\begin{corollary}
Let $q$ be a prime power such that $q\equiv 9\,(\mod{16})$ and $e$ be an arbitrary positive integer. 
Then, there exists a skew Hadamard matrix of order $4(q^e+1)$. 
\end{corollary}
\section{Evaluation of Gauss sums}
Let $\F_q$ be the finite field of order $q=p^r$ with $p$ a prime and 
$\F_q^\ast$ be the multiplicative group of $\F_q$. 
Let $\gamma$ be a primitive element of $\F_q$. For a positive integer 
$N$ dividing $q-1$, define \[
C_i^{(N,q)}=\gamma^i\langle \gamma^N\rangle, \, \, \, \, i=0,1,\ldots,N-1, 
\]
which are called {\it cyclotomic classes} of order $N$. 
We will need to compute additive character values of a union of some cyclotomic classes of order $N$. So we introduce additive characters of finite fields below.

For a positive integer $k$, let $\zeta_k$ be a complex primitive $k$th root of unity. 
Define $\psi_{\F_q}$: $\F_q\to \C^\ast$ by 
\[
\psi_{\F_q}(x)=\zeta_p^{\Tr_{q/p}(x)}, 
\]
where $\Tr_{q/p}(x)$ is the trace function from $\F_q$ to $\F_p$. The map $\psi_{\F_q}$ is a character of the additive group of $\F_q$, and it is called the {\it canonical} additive character of $\F_q$. 
\begin{definition}{\em 
For a multiplicative character 
$\chi$  and the canonical 
additive character $\psi_{\F_q}$ of $\F_q$, the {\it Gauss sum} $G_{q}(\chi)$ of $\F_q$ is defined by
\[
G_q(\chi)=\sum_{x\in \F_q^\ast}\chi(x)\psi_{\F_q}(x).
\]}
\end{definition}
For a multiplicative character $\chi$ of order $N$ of $\F_q$ and $x\in \F_q^\ast$, by the orthogonality of characters~\cite[p.~195, (5.17)]{LN97}, 
the character value of $C_{i}^{(N,q)}$ can be expressed in terms of Gauss 
sums as follows: 
\begin{equation}\label{orth}
\psi_{\F_q}(C_i^{(N,q)})=\frac{1}{N}\sum_{j=0}^{N-1}G_q(\chi^{j})\chi^{-j}(\gamma^i), \; 0\le i\le N-1. 
\end{equation}
We list some basic properties of Gauss sums below, which will be used in Section 3.
\begin{lemma}\label{basic}
The Gauss sums $G_{q}(\chi)$ satisfy the following: 
\begin{enumerate}
\item[1. ] $G_q(\chi)\overline{G_q(\chi)}=q$ if $\chi$ is nontrivial;
\item[2. ] $G_q(\chi^{-1})=\chi(-1)\overline{G_q(\chi)}$;
\item[3. ] $G_q(\chi)=-1$ if $\chi$ is trivial;
\item[4. ] $G_q(\chi^p)=G_q(\chi)$.
\end{enumerate}
\end{lemma}

We will need the  {\it Davenport-Hasse lifting formula}, which is stated below.  
\begin{theorem}\label{thm:lift}
{\em (\cite[Theorem~11.5.2]{BEW})}
Let $\chi'$ be a nontrivial multiplicative character of $\F_{q}$ and 
let $\chi$ be the lift of $\chi'$ to $\F_{q^{f}}$, i.e., $\chi(x)=\chi'(x^{\frac{q^f-1}{q-1}})$ for $x\in \F_{q^{f}}$, where $f\geq 2$ is an integer. Then 
\[
G_{q^{f}}(\chi)=(-1)^{f-1}(G_{q}(\chi'))^f. 
\]
\end{theorem}
The following theorem is often referred to as the  {\it Davenport-Hasse product formula}.  
\begin{theorem}
\label{thm:Stickel2}{\em (\cite[Theorem~11.3.5]{BEW})}
Let $\eta$ be a multiplicative character of order $\ell>1$ of  $\F_{q}$. For  every nontrivial multiplicative character $\chi$ of $\F_{q}$, 
\[
G_{q}(\chi)=\frac{G_{q}(\chi^\ell)}{\chi^\ell(\ell)}
\prod_{i=1}^{\ell-1}
\frac{G_{q}(\eta^i)}{G_{q}(\chi\eta^i)}. 
\]
\end{theorem}
In the rest of this paper, we always assume that $q$ is a prime power such that $q\equiv 2^u+1\,(\mod{2^{u+1}})$ with $u\ge 2$. Fix $N=2^t$ with $t\ge u+1$ and put 
$f:=2^{t-u}$. 
Then $q$ has order $f$ modulo $N$; that is, $q^f\equiv 1\pmod N$, and $q^{2^{t-u-1}}\not\equiv 1\pmod N$.

The following  is our main theorem in this section. 
\begin{theorem}\label{prop:Gaussreduc}
Let  $\chi_{N}$ be a multiplicative character of order $N$ of $\F_{q^f}$ and $\omega$ be a primitive element of $\F_{q^f}$. Furthermore, 
let $\chi_{2^u}'$ be a multiplicative character of order $2^u$ of $\F_{q}$ such that 
$\chi_{N}^{2^{t-u}}$ is the lift of $\chi_{2^u}'$, i.e., $\chi_{N}^{2^{t-u}}(\omega)=\chi_{2^u}'(\omega^{\frac{q^f-1}{q-1}})$, and let  
$\eta'$ be the quadratic character of $\F_q$. 
Then, 
there exists 
$\epsilon\in \langle \zeta_{2^u}\rangle$ such that 
\[
G_{q^f}(\chi_{N})=\epsilon q^{f/2-1}G_{q}(\chi_{2^u}') G_{q}(\eta'). 
\]
\end{theorem}
\begin{proof}
By the Davenport-Hasse product formula (Theorem~\ref{thm:Stickel2}), we have 
\begin{equation}\label{eq:prodG}
G_{q^f}(\chi_{N})=\frac{G_{q^f}(\chi_N^f)}{\chi_N^f(f)}\prod_{i=1}^{f-1}\frac{G_{q^f}(\chi_f^i)}{G_{q^f}(\chi_N\chi_f^i)}, 
\end{equation}
where $\chi_f$ is a fixed multiplicative character of order $f=2^{t-u}$ of $\F_{q^f}$. We can write $\chi_f=\chi_N^{2^u\ell}$ for some odd $\ell$. Then $\chi_N\chi_f=\chi_N^{2^u\ell +1}$. We claim that for any $\ell$, $0\leq \ell\leq 2^{t-u}-1$, $2^u\ell +1 \in \langle q\rangle \pmod N$. This can be seen as follows. Noting that $q$ has order $f=2^{t-u}$ modulo $N$, and $q^i\equiv 1\pmod {2^u}$ for any $i$, we see that  $2^u\ell +1 \in \langle q\rangle \pmod N$ for all $\ell=0,1,\ldots ,2^{t-u}-1$. Now from Property (4) of Lemma~\ref{basic}, it follows that
$$G_{q^f}(\chi_{N})=G_{q^f}(\chi_N\chi_f)=\cdots =G_{q^f}(\chi_N\chi_f^{f-1}).$$
Also from Properties (1) and (2) of Lemma~\ref{basic}, we have 
$G_{q^f}(\chi_f^i)G_{q^f}(\chi_f^{f-i})= q^f$ for $i=1,2,\ldots,f/2-1$. 
Substituting these into \eqref{eq:prodG}, we obtain 
\begin{equation}\label{eq:simprodG}
G_{q^f}(\chi_{N})^{f}=\chi_N^{-f}(f) q^{f(f/2-1)}G_{q^f}(\chi_{N}^f)G_{q^f}(\eta), 
\end{equation}
where $\eta$ is the quadratic character of $\F_{q^f}$. 
Next applying the Davenport-Hasse lifting formula (Theorem~\ref{thm:lift}) to the right hand side of (\ref{eq:simprodG}), we have 
\[
G_{q^f}(\chi_{N})^{f}=\chi_N^{-f}(f) q^{f(f/2-1)}G_{q}(\chi_{2^u}')^f G_{q}(\eta')^f. 
\]
Hence, there exists 
$\epsilon\in \langle \zeta_N\rangle$ such that 
\begin{equation}\label{eq:relation_Gauss}
G_{q^f}(\chi_{N})=\epsilon q^{f/2-1}G_{q}(\chi_{2^u}') G_{q}(\eta'). 
\end{equation}
Define $\tau\in \Gal(\Q(\zeta_p,\zeta_N)/\Q(\zeta_p))$ by 
$\tau(\zeta_{pN})=\zeta_{pN}^{N\ell+q}$, where $\ell$ is the inverse of 
$N$ modulo $p$. Note that $N\ell +q\equiv q \pmod N$ and $N\ell +q\equiv 1\pmod p$. Applying $\tau$ to $G_{q^f}(\chi_{N})/G_{q}(\chi_{2^u}') G_{q}(\eta')$, 
we have 
\begin{align*}
\tau\left(\frac{G_{q^f}(\chi_{N})}{G_{q}(\chi_{2^u}') G_{q}(\eta')}\right)
=&\, \tau\left(\frac{\sum_{x\in \F_{q^f}}\psi_{\F_{q^f}}(x)\chi_N (x)}{
\left(\sum_{x\in \F_{q}}\psi_{\F_{q}}(x)\chi_{2^u}' (x)\right)
\left(\sum_{x\in \F_{q}}\psi_{\F_{q}}(x)\eta' (x)\right)}\right)\\
=&\, \frac{\sum_{x\in \F_{q^f}}\psi_{\F_{q^f}}(x)\chi_N^q (x)}{
\left(\sum_{x\in \F_{q}}\psi_{\F_{q}}(x){\chi'}_{2^u}^q (x)\right)
\left(\sum_{x\in \F_{q}}\psi_{\F_{q}}((x){\eta'}^q (x)\right)}\\
=&\, \frac{G_{q^f}(\chi_{N}^q)}{G_{q}({\chi'}_{2^u}^q) G_{q}({\eta'}^q)}
= \frac{G_{q^f}(\chi_{N})}{G_{q}(\chi_{2^u}') G_{q}(\eta')}. 
\end{align*}
This shows that $\epsilon$ is invariant under the action of $\tau$. It follows that $\epsilon \in \langle \zeta_{2^u}\rangle$. The proof of the proposition is now complete. 
\end{proof} 
\begin{remark}\label{rem:bbb}
Put $\gamma:=\omega^{\frac{q^f-1}{q-1}}$, which is a primitive element of $\F_q$. 
In the proposition above, $\epsilon$ has the form $\epsilon=\chi'_{2^u}(\gamma^{-b})$ for some $b\in \{0,1,\ldots,2^u-1\}$. 
\end{remark}
\section{Skew Hadamard difference families in finite fields}
We retain the same assumptions on $q$, $u$, $N$ and $f$ as specified in Section 2. Define subsets in $\F_{q}$ and $\F_{q^f}$, respectively, by
\begin{equation}\label{def:blocks1}
B_h:=\bigcup_{j=0}^{2^{u-1}-1}C_{h+j}^{(2^u,q)}\subset \F_{q}, \, \, \, h=0,1,\ldots,2^u-1, 
\end{equation}
and 
\begin{equation}\label{def:blocks2}
D_h:=\bigcup_{j=0}^{N/2-1}C_{h+j}^{(N,q^f)}\subset \F_{q^f}, \, \, \, h=0,1,\ldots,N-1. 
\end{equation}
We list some important properties of $B_h$ and $D_h$  below: 
\begin{itemize}
\item[(i)] $-B_h=B_{h+2^{u-1}}$, $B_h\cap -B_h=\emptyset$, $|B_h|=(q-1)/2,$ and  $B_h\cup -B_h=\F_q^\ast$;
\item[(ii)] $-D_h=D_{h+2^{t-1}}$, $D_h\cap -D_h=\emptyset$, $|D_h|=(q^f-1)/2$, and  $D_h\cup -D_h=\F_{q^f}^\ast$. 
\end{itemize}
Hence, both $B_h$ and $D_h$ are skew. 
\subsection{Character values of $D_h$}
We express the (additive) character values of $D_h$ in terms of those of $B_h$. 
\begin{proposition}\label{prop:chararedu}
For $a=0,1,\ldots, N-1$ and $h=0,1,\ldots, N-1$, we have
\[
\psi_{\F_{q^f}}(\omega^a D_h)=\frac{-1+q^{f/2-1}G_q(\eta')}{2}+q^{f/2-1}G_q(\eta')\psi_{\F_q}(\gamma^{b+j_{a,h}}B_0), 
\]
where $j_{a,h}$ is uniquely determined as 
$j_{a,h}=(a+h+j)/2^{t-u}$ for some $j\in \{0,1,\ldots,2^{t-u}-1\}$ such that $2^{t-u}\,|\,(a+h+j)$, and $b$ is given in Remark~\ref{rem:bbb}.
\end{proposition}
\begin{proof}
By (\ref{orth}), we have
\begin{equation}\label{eq:chara}
\psi_{\F_{q^f}}(\omega^a D_h)=\frac{1}{N}\sum_{j=0}^{N/2-1}\sum_{i=0}^{N-1}
G_{q^f}(\chi_N^i)\chi_N^{-i}(\omega^{a+h+j}). 
\end{equation} 
Since $\sum_{j=0}^{N/2-1}\chi_N^{-i}(\omega^{a+h+j})=0$ if $i$, $0\leq i\leq N-1$, is a nonzero even integer, continuing from \eqref{eq:chara}, we have 
\begin{equation}\label{eq:chara2}
\psi_{\F_{q^f}}(\omega^a D_h)=-\frac{1}{2}+\frac{1}{N}\sum_{j=0}^{N/2-1}\sum_{i: \mbox{\tiny{odd}}}
G_{q^f}(\chi_N^i)\chi_N^{-i}(\omega^{a+h+j}). 
\end{equation}
For each odd $i$, $0\leq i\leq N-1$, write 
$i=2^u i_1+i_2$ with $i_1\in \{0,1,\ldots,2^{t-u}-1\}$ and $i_2\in \{1,3,\ldots,2^u-1\}$. Note that since $2^u\ell +1\in \langle q\rangle \pmod N$, we have $G_{q^f}(\chi_N^{2^u i_1+i_2})=G_{q^f}(\chi_N^{2^u i_1'+i_2})$ for 
any $i_1,i_1'=0,1,\ldots,2^{t-u}-1$.  It follows that 
\begin{equation}\label{eq:Gaussto1}
\sum_{j=0}^{N/2-1}\sum_{i: \mbox{\tiny{odd}}}
G_{q^f}(\chi_N^i)\chi_N^{-i}(\omega^{a+h+j})=
\sum_{i_2: \mbox{\tiny{odd}}}G_{q^f}(\chi_N^{i_2})\sum_{j=0}^{N/2-1}
\sum_{i_1=0}^{2^{t-u}-1}\chi_N^{-2^ui_1-i_2}(\omega^{a+h+j}). 
\end{equation}
Let $J:=\{a+h+j\,|\,a+h+j\equiv 0\,(\mod{2^{t-u}}), j=0,1,\ldots,N/2-1\}$. 
There is a unique $j \in \{0,1,\ldots,2^{t-u}-1\}$ such that 
$a+h+j\equiv 0\,(\mod{2^{t-u}})$; for such a $j$, write $a+h+j=2^{t-u}j_{a,h}$. Then, we have 
$J=\{2^{t-u}(j_{a,h}+j')\,|\,j'=0,1,\ldots,2^{u-1}-1\}$. 
Hence, 
\begin{align*}
\sum_{j=0}^{N/2-1}
\sum_{i_1=0}^{2^{t-u}-1}\chi_N^{-2^ui_1-i_2}(\omega^{a+h+j})=&\,
\sum_{j=0}^{N/2-1}\chi_N^{-i_2}(\omega^{a+h+j})
\sum_{i_1=0}^{2^{t-u}-1}\chi_N^{-2^ui_1}(\omega^{a+h+j})\\
=&\,
2^{t-u}\sum_{j'=0}^{2^{u-1}-1}\chi_N^{-i_2}(\omega^{2^{t-u}(j_{a,h}+j')})
= 
2^{t-u}\sum_{j'=0}^{2^{u-1}-1}{\chi'}_{2^u}^{-i_2}(\gamma^{j_{a,h}+j'}). 
\end{align*}
Applying Theorem~\ref{prop:Gaussreduc} and Remark~\ref{rem:bbb}, continuing from \eqref{eq:Gaussto1}, we have 
\begin{align}\label{eq:transf1}
& \, \sum_{i_2: \mbox{\tiny{odd}}}G_{q^f}(\chi_N^{i_2})\sum_{j=0}^{N/2-1}
\sum_{i_1=0}^{2^{t-u}-1}\chi_N^{-2^ui_1-i_2}(\omega^{a+h+j})\nonumber\\
=\nonumber&\,
q^{f/2-1}G_q(\eta')\sum_{i_2: \mbox{\tiny{odd}}}G_{q}({\chi'}_{2^u}^{i_2})\sum_{j=0}^{N/2-1}
\sum_{i_1=0}^{2^{t-u}-1}\chi_N^{-2^ui_1-i_2}(\omega^{a+h+j}){\chi'}_{2^u}^{-i_2}(\gamma^b)\\
=&\,
2^{t-u}q^{f/2-1}G_q(\eta')\sum_{j'=0}^{2^{u-1}-1}\sum_{i_2: \mbox{\tiny{odd}}}G_{q}({\chi'}_{2^u}^{i_2}){\chi'}_{2^u}^{-i_2}(\gamma^{b+j_{a,h}+j'}).
\end{align}
Since $\sum_{j'=0}^{2^{u-1}-1}{\chi'}_{2^u}^{-i}(\gamma^{b+j_{a,h}+j'})=0$ for any nonzero even $i$, we have 
\begin{equation}\label{eq:trans2}
\sum_{j'=0}^{2^{u-1}-1}\sum_{i_2: \mbox{\tiny{odd}}}G_{q}({\chi'}_{2^u}^{i_2}){\chi'}_{2^u}^{-i_2}(\gamma^{b+j_{a,h}+j'})=2^{u-1}+2^u\psi_{\F_q}(\gamma^{b+j_{a,h}}B_0). 
\end{equation}
Thus, by combining \eqref{eq:chara2}, \eqref{eq:Gaussto1}, \eqref{eq:transf1} and \eqref{eq:trans2}, we obtain
\begin{align*}
\psi_{\F_{q^f}}(\omega^a D_h)=
\frac{-1+q^{f/2-1}G_q(\eta')}{2}+q^{f/2-1}G_q(\eta')\psi_{\F_q}(\gamma^{b+j_{a,h}}B_0). 
\end{align*}
This completes the proof of the proposition. 
\end{proof}

\begin{remark}\label{rem:hhh}
For  $j_{a,h}$ defined in Proposition~\ref{prop:chararedu}, we have $j_{a,h+2^{t-u}\ell}=j_{a,h}+\ell$ for any $\ell\in \Z$. 
\end{remark}
\begin{corollary}\label{coro:skewtimes}
For $a=0,1,\ldots,N-1$ and $h=0,1,\ldots ,N-1$, we have
\[
\psi_{\F_{q^f}}(\omega^a D_h)\overline{\psi_{\F_{q^f}}(\omega^a  D_h)}=\frac{1-q^{f-1}}{4}+q^{f-1}
\psi_{\F_q}(\gamma^{b+j_{a,h}}B_0)\overline{\psi_{\F_q}(\gamma^{b+j_{a,h}}B_0)}. 
\]
\end{corollary}
\begin{proof}
By Proposition~\ref{prop:chararedu}, we have 
\begin{align}
\psi_{\F_{q^f}}(\omega^a D_h)\overline{\psi_{\F_{q^f}}(\omega^a D_h)}
=&\,\left(\frac{-1+q^{f/2-1}G_q(\eta')}{2}+q^{f/2-1}G_q(\eta')\psi_{\F_q}(\gamma^{b+j_{a,h}}B_0)\right)\label{eq:prod_conj}\\
&\hspace{2cm}\times \left(\overline{\frac{-1+q^{f/2-1}G_q(\eta')}{2}+q^{f/2-1}G_q(\eta')\psi_{\F_q}(\gamma^{b+j_{a,h}}B_0)}\right).\nonumber
\end{align}
Here, $G_q(\eta') \in \R$ since $q\equiv 1\,(\mod{4})$. 
Furthermore, note that 
\[
\psi_{\F_q}(\gamma^{b+j_{a,h}}B_0)+\overline{\psi_{\F_q}(\gamma^{b+j_{a,h}}B_0)}=\psi_{\F_q}(\gamma^{b+j_{a,h}}B_0)+\psi_{\F_q}(-\gamma^{b+j_{a,h}}B_0)=-1.
\]
Now, continuing from \eqref{eq:prod_conj}, we have
\[
\psi_{\F_{q^f}}(\omega^aD_h)\overline{\psi_{\F_{q^f}}(\omega^a D_h)}
= \frac{1-q^{f-2}G_q(\eta')^2}{4}+q^{f-2}G_q(\eta')^2 
\psi_{\F_q}(\gamma^{b+j_{a,h}}B_0)\overline{\psi_{\F_q}(\gamma^{b+j_{a,h}}B_0)}. 
\]
Finally, by using $G_q(\eta')^2=q$, we obtain the assertion of the corollary. 
\end{proof}
\subsection{Construction of skew Hadamard difference families}
We begin by stating the following lemma which is well known in the theory of difference families. We refer the reader to \cite{bjl} for a proof of the lemma.

\begin{lemma}\label{rem:difffa}
Let  $(G, +)$ be a finite abelian group, and $E_i$, $1\leq i\leq\ell$, be $k$-subsets of $G$.  
Then, the following are equivalent: 
\begin{itemize}
\item[(1)] The family $\{E_i\,|\,i=1,2,\ldots,\ell\}$ is a difference family in $G$;
\item[(2)] For any $a\in G\setminus \{0_G\}$, $\sum_{i=1}^\ell |E_i\cap (E_i+a)|$ is a constant 
not depending on $a$; 
\item[(3)] For any nontrivial character $\psi$ of $G$, 
$\sum_{i=1}^\ell \psi(E_i)\overline{\psi(E_i)}=k\ell-\lambda$ for some $\lambda\in \N$. 
\end{itemize}
\end{lemma}

Now we define two families of subsets of $\F_q$ and $\F_{q^f}$, respectively, by 
\begin{equation}\label{eq:def_B}
{\mathcal B}:=\{B_{i}\,|\,i=0,1,\ldots,2^{u-1}-1\}
\end{equation}
and 
\begin{equation}\label{eq:def_D}
{\mathcal D}:=\{D_{2^{t-u}\ell}\,|\,\ell=0,1,\ldots,2^{u-1}-1\}.
\end{equation}

\begin{lemma}\label{lem:DFexist}
The family ${\mathcal B}$ is a skew Hadamard difference family in $(\F_{q},+)$. 
\end{lemma}

\begin{proof}
We show that $\sum_{i=0}^{2^{u-1}-1} |B_i\cap (B_i+a)|$ is a constant 
not depending on $a\in \F_{q}^\ast$. Then, by Lemma~\ref{rem:difffa}~(2),  ${\mathcal B}$ is a skew Hadamard difference family in $(\F_{q},+)$. Since $-B_i=B_{i+2^{u-1}}$, we have 
$|B_i\cap (B_i+a)|=|B_{i+2^{u-1}}\cap (B_{i+2^{u-1}}+a)|$. It follows that
\begin{align}
\sum_{i=0}^{2^{u-1}-1}|B_i\cap (B_i+a)|=&\,\frac{1}{2}\sum_{i=0}^{2^u-1}
|B_i\cap (B_i+a)|\label{eq:cyclo}\\
=&\,\frac{1}{2}\sum_{j=0}^{2^{u-1}-1} \sum_{h=0}^{2^{u-1}-1}\sum_{i=0}^{2^u-1}
\left|C_{i+j}^{(2^u,q)}\cap \left(C_{i+h}^{(2^u,q)}+a\right)\right|. \nonumber
\end{align}
Since the equation $\gamma^{i+j} x=\gamma^{i+h}y+a$ can be rewritten as $\gamma^{j-h} xy^{-1}=a\gamma^{-i-h}y^{-1}+1$ with $x,y\in C_{0}^{(2^u,q)}$, the right-hand side of the second equality in \eqref{eq:cyclo} is equal to 
\[
\frac{1}{2}\sum_{j=0}^{2^{u-1}-1} \sum_{h=0}^{2^{u-1}-1}\sum_{i=0}^{2^u-1}
\left|C_{j-h}^{(2^u,q)}\cap \left(a\cdot C_{-i-h}^{(2^u,q)}+1\right)\right|=\frac{1}{2}\sum_{j=0}^{2^{u-1}-1} \sum_{h=0}^{2^{u-1}-1}
\left|C_{j-h}^{(2^u,q)}\cap \left(\F_{q}^\ast+1\right)\right|, 
\] 
which is a constant not depending on $a\in \F_{q}^\ast$. 
\end{proof}

\begin{theorem}\label{main:thm}
The family ${\mathcal D}$ is a skew Hadamard difference family in $(\F_{q^f},+)$. 
\end{theorem}
\begin{proof}
By Proposition~\ref{coro:skewtimes}, we have 
\begin{align}\label{eq:DB}
&\, \sum_{\ell=0}^{2^{u-1}-1}\psi_{\F_{q^f}}(\omega^a D_{2^{t-u}\ell})\overline{\psi_{\F_{q^f}}( \omega^a D_{2^{t-u}\ell})}\\
=&\, \frac{2^{u-1}(1-q^{f-1})}{4}+q^{f-1}\sum_{\ell=0}^{2^{u-1}-1}
\psi_{\F_q}(\gamma^{b+j_{a,2^{t-u}\ell}}B_0)\overline{\psi_{\F_q}(\gamma^{b+j_{a,2^{t-u}\ell}}B_0)}. \nonumber 
\end{align}
By Remark~\ref{rem:hhh}, the right-hand side of  
\eqref{eq:DB} can be rewritten as 
\[
\frac{2^{u-1}(1-q^{f-1})}{4}+q^{f-1}\sum_{\ell=0}^{2^{u-1}-1}
\psi_{\F_q}(\gamma^{b+j_{a,0}}B_\ell)\overline{\psi_{\F_q}(\gamma^{b+j_{a,0}}B_\ell)}. 
\]
By Part~(3) of Lemma~\ref{rem:difffa}, 
${\mathcal D}$ is a skew Hadamard difference family in $(\F_{q^f},+)$ if and only if ${\mathcal B}$ is a skew Hadamard difference family in $(\F_{q},+)$. Now the conclusion of the theorem follows from Lemma~\ref{lem:DFexist}. 
\end{proof}

Noting that  $q\equiv 2^u+1\,(\mod{2^{u+1}})$ if and only if $q^s\equiv 2^u+1\,(\mod{2^{u+1}})$ for any odd positive integer $s$, we see that Theorem~\ref{thm:main2} follows from Theorem~\ref{main:thm}. 


\end{document}